\documentclass[11pt]{amsart}
\usepackage{geometry}                
\usepackage{graphicx}
\usepackage{amssymb}
\usepackage{epstopdf}
\newtheorem{theorem}{Theorem}[section]
\newtheorem{lemma}{Lemma}[section]
\newtheorem{corollary}{Corollary}[section]

\newtheorem{definition}{Definition}[section]

\newtheorem{question}{Question}[section]
\numberwithin{equation}{section}

\def\Z{\Bbb Z}

\def\R{\Bbb R}

\def\d{\partial}
\def\s{\sigma}
\def\e{\epsilon}

\def\b{\beta}
\def\g{\gamma}
\def\om{\omega}

\title{On Holographic Structures, Traversing Flows, and Exotic Spheres}
\author{Gabriel Katz}
\address{MIT, Department of Mathematics, 77 Massachusetts Ave., Cambridge, MA 02139, U.S.A.}
\email{gabkatz@gmail.com}

\begin{document}

\maketitle

\begin{abstract} Any traversally generic vector flow on a compact manifold $X$ with boundary leaves some residual structure on its boundary $\d X$. A part of this structure is the flow-generated causality map $C_v$, which takes a region of $\d X$ to the complementary region. By the Holography Theorem from \cite{K4}, the map $C_v$ allows to reconstruct $X$ together with the unparametrized flow. The reconstruction is a manifestation of holographic description of the flow. 

In the paper, we introduce and study the holographic structures on a given closed manifold $Y$, which mimics $\d X$. We generalize the Holography Theorem so that is stated in terms of fillable holographic structures on $Y$. Such structures are intimately linked with traversally generic vector flows on manifolds $X$ whose boundary is $Y$. We conclude with few observations about the richness of holographic structures on smooth exotic spheres.
\end{abstract}


\section{Basic facts about traversing vector fields and the holography they generate }

We start with a short review of different types of vector fields on smooth manifolds with boundary. 

Let $X$ be a connected compact smooth $(n+1)$-dimensional manifold with boundary. A smooth vector field $v$ on $X$ is called \emph{traversing} if it admits a Lyapunov function $F: X \to \R$ such that $dF(v) > 0$ everywhere in $X$. The trajectories of a traversing vector field are homeomorphic to closed segments or to singletons (see \cite{K1}).  

In \cite{K2}, we introduced a class of vector fields on $X$ which we call \emph{traversally generic}.  By Theorem 3.5 from \cite{K2}, the traversally generic vector fields form an open and dense set in the space of all traversing fields. Any trajectory $\g$ of a traversally generic vector field generates a finite sequence $\omega = (\omega_1, \dots , \omega_q)$
of natural numbers,  the entries of the sequence correspond to $v$-ordered points of the set $\g \cap \d X$.  Each point $x \in \g \cap \d X$ contributes \emph{the multiplicity of tangency} $j(x)$ of $\g$ to the boundary $\d X$ at $x$. The ordered list $\omega_\g$ of these multiplicities is \emph{the combinatorial type} of $\g$. For traversally generic vector fields, the combinatorial types $\omega_\g$ of their trajectories belong to an universal ($X$-independent) poset $\mathbf\Omega^\bullet_{'\langle n]}$ (see \cite{K3} for its definition and proprties). Figure 1 shows the poset $\mathbf\Omega^\bullet_{'\langle 2]}$.

Remarkably, for the traversally generic vector fields, the combinatorial type $\omega_\g$ determines the $v$-flow in the vicinity of $\g \subset X$. 

\begin{figure}[ht]\label{fig1.0}
\centerline{\includegraphics[height=2in,width=3.5in]{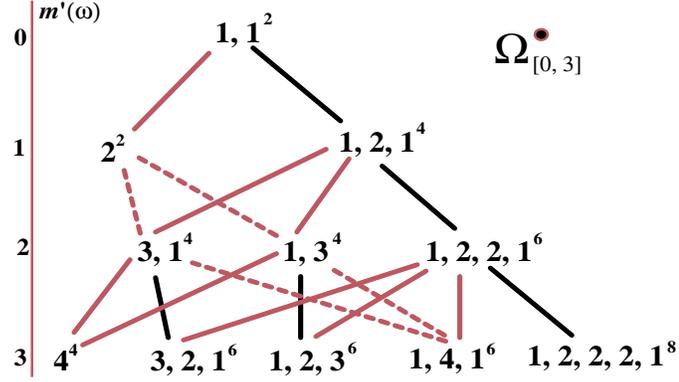}}
\bigskip
\caption{\small{The poset $\mathbf\Omega^\bullet_{'\langle 2]}$. The lines (both bold and dotted) that connect $\om$'s show the partial order in $\mathbf\Omega^\bullet_{'\langle 2]}$. 
The reduced norms $|\sim|'$ are plotted on the vertical axis, marked as $m'(\om)$. The upper indexes of $\om$'s represent their norms $|\sim|$. If two elements of $\mathbf\Omega^\bullet_{'\langle 2]}$ are connected by an edge or their chain, the element with a greater reduced norm is smaller. }} 
\end{figure}

Let $\g_x$ denote the $v$-trajectory through $x \in X$. Any traversing vector field $v$ on $X$ produces a so called \emph{causality map} $C_v$ which takes a portion $\d_1^+X(v)$ of the boundary $\d X$ to the closure $\d_1^-X(v)$ of the complementary portion. Here $\d_1^\pm X(v)$ stands for the locus in $\d X$, where $v$ is directed inward/outward of $X$ or is tangent  to $\d X$. By the definition, $C_v(x)$ is the point $y \in \d_1^-X(v)$ that resides in $\g_x \cap \d X$ above $x \in \d_1^+X(v)$. When no such $y$ exists, we put $C_v(x) = x$.  We stress that, in general, $C_v$ is a discontinuous map. \smallskip

The \emph{boundary generic} (see \cite{K1} for the definition) and traversing vector fields form a wider class than the traversally generic vector fields, the focus of this paper.\smallskip

For the reader convenience, let us state Theorem 3.1 from \cite{K4}. 

\begin{theorem}\label{th1.1}{\bf (The Holography Theorem)}
Let $X_1, X_2$ be two smooth compact connected $(n +1)$-manifolds with boundary, equipped with traversing boundary generic fields  $v_1, v_2$, respectively. 
\begin{itemize}
\item Then any smooth diffeomorphism  $\Phi^\d: \d X_1 \to \d X_2$, such that $$\Phi^\d \circ C_{v_1} =  C_{v_2} \circ \Phi^\d,$$ extends to a homeomorphism $\Phi: X_1 \to X_2$ which maps $v_1$-trajectories to $v_2$-trajectories so that the field-induced orientations of trajectories are preserved. The restriction of $\Phi$ to each trajectory is a smooth diffeomorphism. 

If the fields $v_1, v_2$ are traversally generic, then the conjugating homeomorphism $\Phi: X_1 \to X_2$ is a smooth diffeomorphism outside some closed subsets of codimension $3$.\smallskip

\item If each $v_1$-trajectory is either transversal to $\d X_1$ at \emph{some} point, or is simply tangent to $\d X_1$, then the homeomorphism $\Phi$ is a smooth diffeomorphism. In particular, this is the case for any concave vector field $v_1$. \smallskip
\hfill $\diamondsuit$
\end{itemize}
\end{theorem}

The hypotheses in the second bullet of Theorem \ref{th1.1} are perhaps superfluous: we conjecture that the conjugating homeomorphism $\Phi: X_1 \to X_2$ is always a diffeomorphism.\smallskip

The Holography Theorem provides us with a different window on compact smooth manifolds. In particular, it allows for a reconstruction of $X$ from some data on its boundary $\d X$. Thus the term ``holography". 

In particular, given a closed connected smooth $(n+1)$-manifold $Y$, we can delete a smooth $(n+1)$-ball $B$ from $Y$ to form a manifold $X$ with spherical boundary $S^n$. Then there exists a traversally generic vector field $v$ on $X$ such that $\d_1^+X(v) \subset S^n$ is the Northern hemisphere $S^n_+$, $\d_1^-X(v) \subset S^n$ is the Southern hemisphere $S^n_-$ (\cite{K4}). For such $v$, the causality map $C_v: S^n_+ \to S^n_-$ encodes complete topological information about $X$ and thus about $Y$. These facts will motivate us in what follows.

\section{Holographic Structures on Closed Manifolds}\label{section8.5}

Let $\mathcal F(v)$ denote the oriented 1-dimensional foliation, generated by a traversing vector field $v$ on a $(n+1)$-dimensional manifold $X$.  Theorem \ref{th1.1} tells us that the causality map $C_v: \d_1^+X(v) \to \d_1^-X(v)$ allows to restore the pair $(X, \mathcal F(v))$ up to a homeomorphism (often a diffeomorphism) of $X$ which is the identity map on the boundary $\d X$.

The word ``holography" in the name of the theorem is justified, since the information about the $(n+1)$-dimensional bulk $X$ and the $v$-flow on it is recorded via a map of two $n$-dimensional screens, $\d_1^+X(v)$ and $\d_1^-X(v)$. 

The reconstruction procedure $(\d X, C_v) \Rightarrow (X, \mathcal F(v))$ presumes that the manifold $X$ and the traversally generic field $v$ on it, giving rise to the boundary-confined data $(\d X, C_v)$, \emph{do exist}. 

For a traversally generic $v$, thanks to Lemma 3.4 from \cite{K2}, we have ``semi-local" models for the causality map $C_v$ in the vicinity of each $C_v$-``trajectory" $$\g^\d_x =_{\mathsf{def}} \big\{x, C_v(x), C_v(C_v(x)),\, \dots ,\,  C_v^{(d)}(x)\big\} \subset (\d X)^{d +1},$$ where all the iterative $C_v$-images of $x$ are distinct.
\smallskip

However, the question
\smallskip
 
\emph{``Which pairs $(\d X, C_v)$ arise from traversally generic fields $v$ on some manifold $X$?"} 
\smallskip

\noindent deserves attention and animates this paper. \smallskip

For a traversing $v$, we may collapse each trajectory to a singleton, thus generating a quotient map $\Gamma: X \to \mathcal T(v)$. We call the set $\mathcal T(v)$, equipped with the quotient topology, \emph{the trajectory space}. For traversally generic vector fields $v$, the spaces $\mathcal T(v)$ are compact $n$-dimensional $CW$-complexes, stratified by strata which are labeled by the elements $\omega \in \mathbf \mathbf\Omega^\bullet_{'\langle n]}$ (\cite{K4}). Moreover, for a traversally  generic $v$ on a $(n+1)$-dimensional $X$, we can model \emph{locally} the topological types of the $\mathbf \mathbf\Omega^\bullet_{' \langle n]}$-stratified trajectory spaces $\mathcal T(v)$ (\cite{K3}, Theorem 5.3). \smallskip

Again, the question \smallskip

``\emph{Which compact $\mathbf\mathbf\Omega^\bullet_{' \langle n]}$-stratified $CW$-complexes $\mathcal T$ can be produced as trajectory spaces $\mathcal T(v)$ of traversally  generic pairs $(X, v)$?}" 
\smallskip

\noindent remains open. \smallskip

In the definitions below, we will develop a language that is appropriate for answering these questions.  
\smallskip

First, we need to introduce some notations. With any $\omega = (\om_1, \dots, \om_q) \in  \mathbf\Omega^\bullet_{' \langle n]}$, we associate two quantities: the norm
$|\omega| = \sum_i \omega_i$
and the \emph{reduced} norm
$|\omega|' = \sum_i (\omega_i -1).$

The bizarre notation ``$\mathbf\Omega^\bullet_{' \langle n]}$" indicates that this poset is formed from a ``more universal" poset $\mathbf\Omega^\bullet$ by considering only its elements $\om$ whose reduced norms satisfy the constraint $|\om|' \leq n$. 
\smallskip

For any $\omega \in \mathbf\Omega^\bullet_{' \langle n]}$, consider the product $\R\times\R^{|\omega|'}\times\R^{n-|\omega|'} \approx \R^{n+1}$. We denote by $(u, \vec x, \vec y)$ a typical point in the product. We form a cylinder  
$$\Pi_\omega(\e, r) =_{\mathsf{def}} \{(u, \vec x, \vec y):\; \|\vec{x}\| < \e,\,  \|\vec{y}\| < r\} \subset \R\times\R^{|\omega |'}\times\R^{n-|\omega |'}$$ 
whose base is an open polydisk. We denote by $\pi$ the obvious projection of $\R\times\R^{|\omega|'}\times\R^{n-|\omega|'} $ onto $\R^{|\omega|'}\times\R^{n-|\omega|'} \approx \R^n$. 

We  also consider the model real polynomial 
\begin{eqnarray}\label{eq2.0}
P_\omega(u, \vec{x}) =_{\mathsf{def}} \prod_{i\, \in \,  sup(\omega)} \big[(u - i)^{\omega(i)} + \sum_{k = 0}^{\omega(i) -2} x_{i, k}(u - i)^k\big],
\end{eqnarray} 
where the variables $\vec x =_{\mathsf{def}} \{x_{i, k}\}$ are considered in the lexicographic order. 

In what follows, we assume that $\|\vec x\| < \e$, where $\e > 0$ has been chosen so small that all the real roots of $P_\omega(u, \vec{x})$ belong to a union of \emph{disjoint} intervals in $\R$ with the centers at the real roots $\{i\}$.  We use such $\e$ to pick an appropriate cylinder $\Pi_\omega(\e, r)$.  

For a traversally generic vector field on $X$, the polynomial $P_\omega(u, \vec{x})$ is instrumental in modeling the vicinity of each trajectory $\g$ of the combinatorial type $\om$ (\cite{K2}). 

To clarify this claim, let us consider the two types of semi-algebraic sets:  $$\mathsf Z_\omega =_{\mathsf{def}} \{P_\omega(u, \vec{x}) \leq 0\} \cap \Pi_\omega(\e, r) \; \subset \R\times \R^n$$
$$\mathsf W_\omega =_{\mathsf{def}} \{P_\omega(u, \vec{x}) = 0\} \cap \Pi_\omega(\e, r) \; \subset \R\times \R^n.$$
Here $\mathsf Z_\omega$ models $X$ in the vicinity of a $v$-trajectory $\g \subset X$ of the combinatorial type $\om$, $\mathsf W_\omega$ models $\d X$, and the constant vector field $\d_u$ - the traversally generic vector field $v$. In fact, the existence of such special coordinates $(u, \vec x, \vec y)$ in the vicinity of each $v$-trajectory of the combinatorial type $\om$ may serve as a \emph{working definition} of a traversally generic vector field. 

The local types of traversally generic fields on $3$-folds are shown in Figure 1. 
\smallskip

\begin{figure}[ht]\label{fig1.1}
\centerline{\includegraphics[height=4in,width=4.5in]{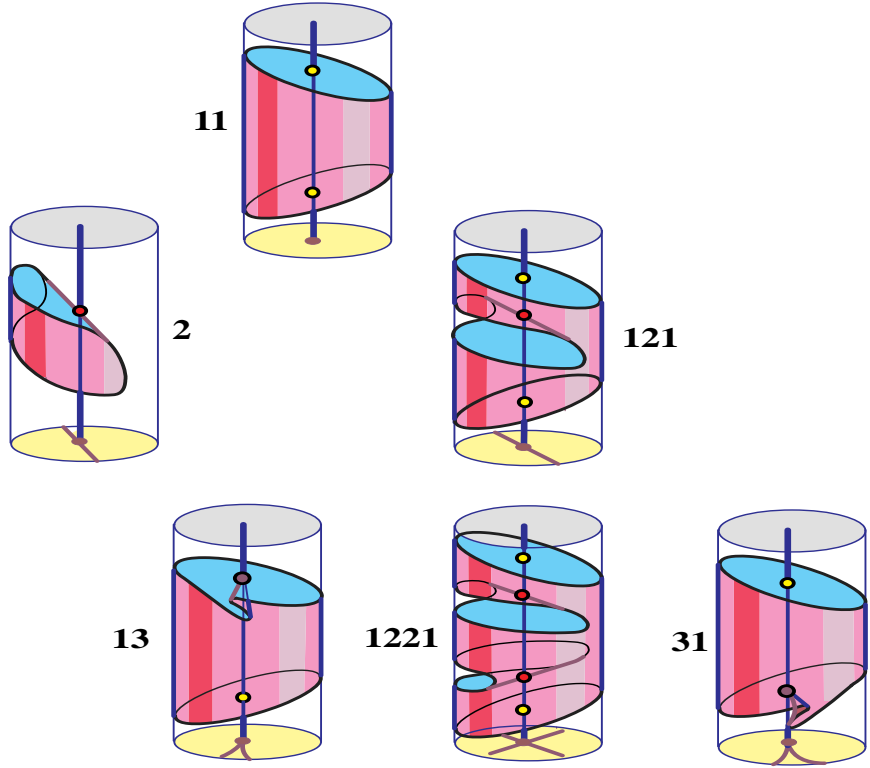}}
\bigskip
\caption{\small{Standard blocks from which holographic structures on closed surfaces are assembled. The rows of integers indicate the combinatorial patterns $\om \in \mathbf\Omega^\bullet_{' \langle 2]}$ of the ``virtual core trajectory", the axis of each cylinder.}} 
\end{figure}

We assume that $\R\times \R^n$ is oriented; from now and on, this orientation is fixed.  The nonsingular hypersurface $\mathsf W_\omega  \subset \R\times \R^n$ \emph{is oriented} with the help of the normal gradient vector field $\nabla (P_\omega(u, \vec{x}))$ and the preferred orientation of $\R\times \R^n$. 
\smallskip

We denote by $\pi_\omega: \mathsf Z_\omega \to \R^n$ the obvious projection. For each vector $\mathbf x =_{\mathsf{def}} (\vec{x}, \vec{y}) \in \R^n$, we consider the connected components $\{\pi^{-1}_{\omega,\, \rho}(\mathbf x)\}_\rho$ of the fiber $\pi^{-1}_\omega(\mathbf x) \subset \mathsf Z_\omega$ over $\mathbf x$ (note that $\pi^{-1}_\omega(\mathbf x)$ does not depend on $\vec{y}$). \smallskip

Next, we will be dealing with a \emph{family} of elements $\{\omega_\b \in  \mathbf\Omega^\bullet_{' \langle n]}\}_\b$, indexed by $\b$, a family of model polynomials $\{P_\b(u_\b, \vec{x}^\b) =_{\mathsf{def}} P_{\omega_\b}(u_\b, \vec{x}^\b) \}_\b$, and a family of appropriate cylinders $\{\Pi_\b =_{\mathsf{def}} \Pi_{\omega_\b}(\e_\b, r_\b)\}_\b$, equipped with the projections $\pi_\b: \Pi_\b \to \R^n$,  induced by the obvious projection $\pi: \R\times\R^n \to \R^n$. 


\begin{definition}\label{def8.5}  Let $Y$ be a closed smooth $n$-manifold. 
A \emph{smooth (oriented) proto-holographic atlas} on $Y$ is an open cover $\{U_\b\}_\b$\footnote{Its elements $U_\b$ are not necessarily connected sets.} of $Y$  together with the following structures: 
\begin{itemize}
\item for each index $\b$, an element $\omega_\b \in \mathbf\Omega^\bullet_{' \langle n]}$ is chosen,
\item for each $\b$, there exists a homeomorphism $h_\b$ from $U_\b$  onto a union of \emph{some} connected components of  the nonsingular real semi-algebraic set
$$\mathsf W_\b =_{\mathsf{def}} \{P_\b(u_\b, \vec{x}^\b) = 0\} \cap \Pi_\b \; \subset \mathsf Z_\b\; \subset \R\times \R^n,$$ 
\item for any pair of indices $\b, \b'$, there is a (orientation preserving) $C^\infty$-diffeomorphism $$h_{\b, \b'}: h_\b(U_\b \cap U_{\b'}) \to h_{\b'}(U_\b \cap U_{\b'})$$ of the hypersurfaces in $\R\times\R^n$ such that $$h_{\b'}|_{U_\b \cap U_{\b'}} =  h_{\b, \b'} \circ h_\b|_{U_\b \cap U_{\b'}},$$
\item  for each $\mathbf x \in \R^n$ and each connected component $\pi_{\b,\, \rho}^{-1}(\mathbf x)$ of the fiber $\pi_\b^{-1}(\mathbf x) \subset \mathsf Z_\b$ of the map $\pi_\b: \mathsf Z_\b \to \R^n$, the diffeomorphism $h_{\b, \b'}$ takes the finite set 
$$h_\b(U_\b \cap U_{\b'}) \cap \pi_{\b,\, \rho}^{-1}(\mathbf x)\;\,\; \text{to} \;\,\; h_{\b'}(U_\b \cap U_{\b'}) \cap \pi_{\b',\, \rho'}^{-1}(\mathbf x')$$ 
and preserves the $u_\b$- and $u_{\b'}$-induced orders of points in the two fibers. Here $\mathbf x' \in \R^n$ and $\pi_{\b',\, \rho'}^{-1}(\mathbf x')$ denotes some connected component of the fiber of $\pi_{\b'}: \mathsf Z_{\b'} \to \R^n$ over $\mathbf x'$.
\end{itemize}

We say that a proto-holographic atlas is \emph{holographic} if all the maps $h_\b: U_\b \to \mathsf W_\b$ are onto. 
\hfill $\diamondsuit$
\end{definition}


\smallskip

\noindent{\bf Remark 2.1.} 
The last bullet of Definition \ref{def8.5} imposes conditions on $h_{\b, \b'}$ which are a relaxed version of the following better looking, but more stringent  conditions: ``$h_{\b, \b'}$ takes each fiber of the obvious projection $$\pi_\b: h_\b(U_\b \cap U_{\b'}) \subset \R\times\R^n \to \R^n$$ to a fiber of $$\pi_{\b'}: h_{\b'}(U_\b \cap U_{\b'}) \subset \R\times\R^n \to \R^n$$ and preserves the natural orders of points in the two fibers".  However, the geometry of traversing flows requires the more relaxed constraint from Definition \ref{def8.5}.  \hfill $\diamondsuit$
\smallskip

\noindent{\bf Example 2.1.} 
\begin{itemize}
\item Here is the most trivial example of a smooth proto-holographic atlas. Pick a smooth structure on $Y$ by specifying an open cover $\{U_\b\}_\b$,  where each $U_\b$ is homeomorphic to an open $n$-ball $D^n$, together with some homeomorphisms $\phi_\b: U_\b \to D^n \subset \R^n$, and the linking diffeomorphisms $$\{\phi_{\b,\b'}: \phi_\b(U_\b \cap U_{\b'}) \to \phi_{\b'}(U_\b \cap U_{\b'})\}_{\b, \b'},$$ which satisfy the usual cocycle conditions. We associate the same element $\omega = (11)$ with each $U_\b$. 
Then let $$P_{\omega_\b}(u_\b, \vec{x}^\b) = (u - 1)(u - 2),$$
 and $$\Pi_\b = \{(u, \vec{y}):\;\|\vec{y}\| < r\}$$ for all $\b$. So $\mathsf W_\b$ is a union of two $n$-balls $$D^n_- = \{(u, \vec{y}): \; u = 1, \,\|\vec{y}\| < r\} \;\, \text{and}\;\, D^n_+ = \{(u, \vec{y}): \; u = 2, \,\|\vec{y}\| < r\},$$ while $\mathsf Z_\b$ is the cylinder $\{(u, \vec{y}):\; 1 \leq u \leq 2,\; \|\vec{y}\| < r\}$. We  identify $D^n$ with $D_-^n$ so that now $\phi_\b = h_\b$ takes $U_\b$ homeomorphically onto $D_-^n$. The verification of all the  properties, required for a proto-holographic atlas, is straightforward.
\smallskip

\item The next example is a variation on the same theme. This time, each element $U_\b$ of a proto-holographic atlas is homeomorphic to a disjoint union of \emph{two} open $n$-balls $D^n$. We associate the same element $\omega = (11)$ with each $U_\b$, but  now $$h_\b: U_\b \to  \mathsf W_\b = D^n_- \coprod D^n_+$$ is surjective---the atlas is holographic. Note that, if there exists a smooth holographic structure on $Y$ which is amenable to such a choice of charts $\{h_\b: U_\b \to D^n_- \coprod D^n_+\}$, then $Y$ is  a covering space with the fiber of cardinality $2$ over some manifold $Y_0$.
\hfill $\diamondsuit$
\end{itemize}

\begin{definition}\label{def8.6} 
\begin{itemize}
\item A proto-holographic atlas $\{V_\s\}_\s$ on $Y$ is called \emph{a refinement} of a proto-holographic atlas $\{U_\b\}_\b$ if each set $V_\s$ is contained in some set $U_\b$, where $\omega_\s \succeq \omega_\b$ in $\mathbf\Omega^\bullet_{'\langle n]}$,  and  the map $h_\s: V_\s \to \mathsf W_\s$ is a restriction of $h_\b: U_\b \to \mathsf W_\b$ to $V_\s$; moreover, the map $h_{\s, \s'}$ is obtained from the map $$h_{\b, \b'}: h_\b(U_\b \cap U_{\b'}) \to h_{\b'}(U_\b \cap U_{\b'})$$ by restricting it to the hypersurface $h_\s(V_\s \cap V_{\s'})$.\smallskip

\item Two smooth proto-holographic atlases on $Y$ are called \emph{equivalent} if they admit a common refinement.\smallskip

\item A \emph{proto-holographic structure} on $Y$ is the equivalence class of  smooth proto-holo-graphic atlases on $Y$. \hfill $\diamondsuit$
\end{itemize}
\end{definition}

Given a proto-holographic/holographic structure $\mathcal H$ on $Y$, any diffeomorphism $\Phi: Y \to Y$ induces the pull-back proto-holographic/holographic  structure $\Phi^\ast\mathcal H$ on $Y$. It is defined  by the formulas: $$\Phi^\ast(U_\b) =_{\mathsf{def}} \Phi^{-1}(U_\b),\quad \Phi^\ast h_\b: \Phi^{-1}(U_\b) \stackrel{\Phi}{\rightarrow}  U_\b  \stackrel{h_\b}{\rightarrow} \mathsf W_\b, \quad\Phi^\ast h_{\b, \b'} = h_{\b,\b'}.$$ 

Thus the group of smooth diffeomorphisms $\mathsf{Diff}(Y)$ acts on the set of all proto-holographic/ holographic structures $\mathcal H(Y)$ on $Y$. Consider $\overline{\mathcal H}(Y)$, the orbit set of this action.
We have little to say about the modular space $\overline{\mathcal H}(Y)$, an interesting object to study...  
\smallskip

\noindent{\bf Remark 2.2.}
The proto-holographic structures do not behave contravariantly under general smooth maps of manifolds. However, a weaker version of  the concept does. We call it appropriately ``\emph{weak proto-holographic structure}". In the second bullet of Definition \ref{def8.5}, we will relax the requirement that $h_\b: U_\b \to \mathsf W_\b$ is a homeomorphism onto a union of some components of the semi-algebraic set  $\mathsf W_\b$; instead, we assume that $h_\b: U_\b \to \mathsf W_\b$ is just a \emph{continuous map}, while keeping the rest of the bullets enforced and adding the cocycle condition 
$$h_{\b'',\b} \circ h_{\b',\b''} \circ h_{\b,\b'}\; |_{h_\b(U_\b \cap U_{\b'} \cap U_{\b''})} = Id$$ for each triple of indices $\b, \b', \b''$.   \hfill $\diamondsuit$
\smallskip

Given a map $\pi: A \to B$ of two sets and a subset $C \subset A$, we denote by $\pi^{!} (C)$ the set $\pi^{-1}(\pi(C)) \subset A$. Evidently $C \subset \pi^! (C)$. When $\pi: A \to B$ is a map of  topological spaces, we denote by $\pi^{!\, \bullet}(C)$ the union of connected components of $\pi^! (C)$ that have nonempty intersections with $C$.

\begin{definition}\label{def8.7} Let $Y$ be a closed smooth $n$-manifold. Consider the semi-algebraic set $$\mathsf Z_\b =_{\mathsf{def}} \{P_\b(u_\b, \vec{x}^\b) \leq 0\} \cap \Pi_\b \; \subset \R\times \R^n,$$
where the polynomial $P_\b$ is of the type (\ref{eq2.0}). We denote by $\pi_\b$ the obvious projection of $\mathsf Z_\b$ on $\R^n$. \smallskip

We say that a smooth holographic atlas $\{U_\b, h_\b\}_\b$ on $Y$ 
is (\emph{orientably}) \emph{fillable} if:
\begin{itemize}
\item for any pair of indices $\b, \b'$, the (orientation-preserving) $C^\infty$-diffeomorphism $$h_{\b, \b'}: h_\b(U_\b \cap U_{\b'}) \to h_{\b'}(U_\b \cap U_{\b'})$$ of the ($\nabla P_\b$-oriented) hypersurfaces in $\R\times\R^n$ extends to a (orientation-preserving) $C^\infty$-diffeomorphism $$H_{\b, \b'}: \mathsf Z_\b^{\b'} \to \mathsf Z_{\b'}^\b$$
of the domains
$$\mathsf Z_\b^{\b'} =_{\mathsf{def}}   \pi_\b^{!\, \bullet}\big(h_\b(U_\b \cap U_{\b'})\big)$$  and $$\mathsf Z_{\b'}^\b =_{\mathsf{def}}  \pi_{\b'}^{!\, \bullet}\big(h_{\b'}(U_\b \cap U_{\b'})\big),$$ 
so that  $H_{\b, \b'}$ maps every connected component of each fiber of the projection $$\pi_\b: \mathsf Z_\b^{\b'} \subset \R\times\R^n \stackrel{\pi}{\rightarrow} \R^n$$ to a connected component of a fiber of $$\pi_{\b'}: \mathsf Z_{\b'}^\b \subset \R\times\R^n \stackrel{\pi}{\rightarrow} \R^n$$ and preserves their orientations, \smallskip

\item  for any triple of indices $\b, \b', \b''$, 
$$H_{\b'',\,\b}\circ H_{\b',\,\b''} \circ H_{\b,\,\b'} \; |_{Z_\b\,  \cap \,  \pi^{!\bullet}_\b(U_\b \cap U_{\b'} \cap U_{\b''})} =  Id.$$
\end{itemize}

We say that a holographic structure $\mathcal H$ on $Y$ is  (\emph{orientably}) \emph{fillable}, if $\mathcal H$ is represented by a holographic atlas which is (orientably) fillable. \hfill $\diamondsuit$ 
\end{definition} 

\noindent {\bf Remark 2.3}
Note that the quotient $\frac{P_{\b'} \circ H_{\b,\b'}}{P_\b}$ of the two functions from Definition \ref{def8.7} must be positive in $\pi_\b^{! \, \bullet}(h_\b(U_\b \cap U_{\b'})).$ 
\hfill $\diamondsuit$
\smallskip

\noindent {\bf Remark 2.4}
Given two proto-holographic/holographic atlases on closed smooth manifolds $Y_1$ and $Y_2$ of equal dimensions, one can form a proto-holographic/holographic atlas on $Y_1 \coprod Y_2$ by considering the obvious  disjoint union of atlases. Similarly, given two holographic fillable atlases on $Y_1$ and $Y_2$, the obvious union of atlases is a holographic  fillable atlas on $Y_1 \coprod Y_2$.\smallskip


These observations \emph{could} lead to the notion of \emph{cobordism} between two \emph{proto}-holographic atlases/structures $\mathcal H_1$ and $\mathcal H_2$ on $Y_1$ and $Y_2$, respectively: it is tempting to define $\mathcal H_1$  to be cobordant to $\mathcal H_2$  if $\mathcal H_1\coprod -\mathcal H_2$ is fillable. Alas, such ``cobordism" is \emph{not} an equivalence relation!

In the future papers, we will introduce a proper notion of a cobordism relation between traversally generic vector fields on manifolds with boundary and will generalize this notion  to cobordisms between holographic structures.  
\hfill $\diamondsuit$

\begin{theorem}\label{th8.7}
\begin{itemize}
\item Any traversally generic vector field $v$ on a (oriented) compact manifold $X$ with boundary gives rise to a smooth holographic (orientably) fillable structure $\mathcal H(v)$ on the manifold $Y = \d_1X$. \smallskip

\item Conversely,  any smooth holographic (orientably) fillable structure $\mathcal H$ on $Y$ arises from a traversally  generic field $v$ on a (oriented) manifold $X$ whose boundary is diffeomorphic to $Y$.
\end{itemize}
\end{theorem}
\begin{proof} Consider a traversally  generic field $v$ on $X$. By Lemma 3.4 from \cite{K2}, 
$X$ has a $v$-adjusted cover $\{V_\b\}_\b$ by connected sets $V_\b$ which admit special coordinates $(u_\b, \vec{x}^\b, \vec{y}^\b)$ as in (\ref{eq2.0}). With each set $V_\b$ we associate the combinatorial type $\omega_\b \in \mathbf\Omega^\bullet_{' \langle n]}$ of the core trajectory $\g_\b \subset V_\b$ (given by the equations $ \vec{x}^\b = 0, \vec{y}^\b = 0$). In these coordinates, $\d X \cap V_\b$ is given by a polynomial equation $P_\b(u_\b, \vec{x}^\b) = 0$ (as in (\ref{eq2.0})) and $V_\b$---by the inequality $P_\b(u_\b, \vec{x}^\b) \leq 0$. These two local descriptions are coupled with the $P_\b$-correlated inequalities $\|\vec{x}^\b\| < \e_\b, \|\vec{y}^\b\| < r_\b$, designed to partition real roots of $u_\b$-polynomials $\{P_\b(u_\b, \vec{x}^\b)\}_{\vec{x}^\b}$ into separate groups, the groups being labeled by the roots of  $P_\b(u_\b, \vec{0})$ (see Lemma 3.4 from \cite{K2}). 

Let $H_\b: V_\b \to \mathsf Z_\b \subset \R^{n+1}$ be the diffeomorphism, given by the local coordinates $(u_\b, \vec{x}^\b, \vec{y}^\b)$. We denote by $h_\b$ its restriction to the boundary portion $U_\b =_{\mathsf{def}} V_\b \cap \d X$. 

By fixing the coordinates $(\vec{x}^\b, \vec{y}^\b)$ and a \emph{marker} $u_\b^\star \in \R$ such that $P_\b(u_\b^\star, \vec{x}^\b) \leq 0$, we are getting the $H_\b$-images of the $v$-trajectories  $\g \subset V_\b$, residing in the set  $$\mathsf Z_\b =_{\mathsf{def}} \{P_\b(u_\b, \vec{x}^\b) \leq 0,\;  \|\vec{x}^\b \| \leq \e_\b,\; \|\vec{y}^\b \| \leq  r_\b\}.$$

We define a diffeomorphism $$H_{\b,\b'}: H_\b(V_\b \cap V_{\b'}) \to H_{\b'}(V_\b \cap V_{\b'})$$ by the formula $H_{\b'}\circ H_\b^{-1} |_{H_\b(V_\b \cap V_{\b'})}$. Due to the nature of the local coordinates $(u_\b, \vec{x}^\b, \vec{y}^\b)$, $H_{\b,\b'}$ takes each connected component of every fiber of the $u$-directed projection $$\pi_\b:  \mathsf Z_\b = H_b(V_\b) \to \R^n$$ to a connected component of a fiber of the  $u$-directed projection $$\pi_{\b'}: \mathsf Z_{\b'} = H_{b'}(V_{\b'}) \to \R^n.$$ It  preserves the $u$-induced orientation of the $\pi$-fibers. 
Note that $$\pi_\b^{!\, \bullet}(h_\b(U_\b \cap U_{\b'})) = H_\b(V_\b \cap V_{\b'})\;\, \text{and} \;\, \pi_{\b'}^{!\, \bullet}(h_{\b'}(U_\b \cap U_{\b'})) = H_{\b'}(V_\b \cap V_{\b'}).$$ 

Evidently, for any triple $\b, \b', \b''$, we obtain the identity $$H_{\b'',\b}\circ H_{\b',\b''} \circ H_{\b,\b'} \; |_{Z_\b \cap H_\b(V_\b \cap V_{\b'} \cap V_{\b''})} =  Id.$$ 

Consider $U_\b =_{\mathsf{def}} V_\b \cap \d X$,  $h_\b =_{\mathsf{def}} H_\b|_{U_\b}$, and $h_{\b,\b'} =_{\mathsf{def}} H_{\b,\b'}|_{h_\b(U_\b \cap U_{\b'})}$. Thus $\{U_\b, h_\b, h_{\b,\b'}\}$ form a holographic atlas on $\d X$, which is fillable. The validation of this claim is on the level of definitions.  
\smallskip

Now we turn to the validation of the second bullet in Theorem \ref{th8.7}. Consider a fillable holographic atlas $\{U_\b, h_\b, h_{\b,\b'}\}_\b$ on $Y$ (as in Definition \ref{def8.7}). In particular, we have the diffeomorphisms $$H_{\b, \b'}: \mathsf Z_\b^{\b'} =_{\mathsf{def}}  \pi_\b^{!\, \bullet}\big(h_\b(U_\b \cap U_{\b'})\big) \to \mathsf Z_{\b'}^\b =_{\mathsf{def}}  \pi_{\b'}^{!\, \bullet}\big(h_{\b'}(U_\b \cap U_{\b'})\big)$$ to play with. \smallskip

Consider the hypersurface $$\mathsf W_\b =_{\mathsf{def}} \{P_\b(u_\b, \vec{x}^\b) = 0\} \cap \Pi_\b \; \subset \R\times \R^n.$$

Let $X$ denote a quotient of the space $\coprod_\b \mathsf Z_\b$ by the following equivalence relation:  $z \in \mathsf Z_\b$ is equivalent to $z' \in \mathsf Z_{\b'}$ if $z  \in  \mathsf Z_\b^{\b'} \; \text{and}\; z' = H_{\b,\b'}(z) \in \mathsf Z_{\b'}^\b.$  The cocycle condition $$H_{\b'',\b}\circ H_{\b',\b''} \circ H_{\b,\b'} \; |_{Z_\b \cap \pi^!_\b(U_\b \cap U_{\b'} \cap U_{\b''})} =  Id$$ 
implies that if $z \sim z'$ and $z' \sim z''$, then $z'' \sim z$.\smallskip

We denote by $q: \coprod_\b \mathsf Z_\b \to X$ the quotient map that takes each point $z$ to its equivalence class $q(z)$. Evidently, for each index $\b$, $q:  \mathsf Z_\b \to q(\mathsf Z_\b)$ is a homeomorphism.

It is on the level of definitions to verify that $X$ is a smooth compact $(n+1)$-manifold with the smooth structure being defined by the atlas $$\{q(\mathsf Z_\b),\; q^{-1}: q(\mathsf Z_\b)  \to \R\times\R^n\}_\b$$ and with the boundary $\d X = \cup_\b\; q(\mathsf W_\b)$ being diffeomorphic to $Y$.

Moreover, $X$ admits a $1$-dimensional foliation $\mathcal F^u$ with oriented leaves. Indeed, each chart $q(\mathsf Z_\b)$ comes equipped with a $1$-dimensional foliation $\mathcal F_\b$, generated as the $q$-image of the connected components of fibers of the projection $\pi_\b: \mathsf Z_\b \to \R^n$. Since the gluing  diffeomorphisms $\{H_{\b,\b'}\}$ take components of the fibers of  $\pi: \mathsf Z_\b^{\b'} \to \R^n$ to components of the fibers of $\pi: \mathsf Z_{\b'}^\b \to \R^n$ and preserve their orientation, these local foliations $\mathcal F_\b$ produce a globally defined foliation $\mathcal F^u$ on $X$ with oriented leaves. We can equip $X$ with a Riemannian metric $g$. Let $v$ be the unit vector field tangent to the leaves of $\mathcal F^u$ and coherent with their orientation. By the very construction of $\mathcal F^u$ from the traversally  generic building blocks $\{\mathsf Z_\b\}_\b$, the field $v$ is traversally  generic. 

Thus, via this construction, we have built the desired pair $(X, v)$ from a given fillable holographic atlas. 
\hfill \end{proof}

\begin{corollary}\label{cor8.13} If a closed smooth manifold $Y$ admits a (orientably) fillable holographic structure, then $Y$ is a boundary of a (orientable) smooth manifold $X$. 
\hfill $\diamondsuit$
\end{corollary} 
\smallskip

Now, we would like to recall briefly the definition and construction of the Morse stratifications (\cite{Mo})
$$\d X = \d_1X(v) \supset \d_2X(v) \supset \dots \supset \d_{n+1}X(v)$$ and 
$$\d_1^+X(v) \supset \d_2^+X(v) \supset \dots \supset \d_{n+1}^+X(v),$$ induced by a boundary generic vector field $v$ on $X$. For a boundary generic $v$,  all the loci $\{\d_jX(v)\}_j$ are closed manifolds, and  the loci $\{\d_j^\pm X(v)\}_j$ are compact manifolds (possibly, with boundary).
 
Let $\d_jX =_{\mathsf{def}} \d_jX(v)$ denote the locus of points $x \in \d X$ such that the multiplicity of the $v$-trajectory $\g_x$ through $x$ at $x$ is greater than or equal to $j$. 

We may embed the compact manifold $X$ into an open manifold $\hat X$ of the same dimension, so that $v$ extends smoothly to a non-vanishing vector field $\hat v$ in $\hat X$. We treat the extension $(\hat X, \hat v)$ as ``a germ at $(X, v)$". 

With the help of embedding, the locus $\d_jX =_{\mathsf{def}} \d_jX(v)$ has a description in terms of an auxiliary function $z: \hat X \to \R$ that satisfies the following three properties:
\begin{eqnarray}\label{eq2.3}
\end{eqnarray}

\begin{itemize}
\item $0$ is a regular value of $z$,   
\item $z^{-1}(0) = \d X$, and 
\item $z^{-1}((-\infty, 0]) = X$. 
\end{itemize}

In terms of $z$, the locus $\d_jX =_{\mathsf{def}} \d_jX(v)$ is defined by the equations: 
$$\{z =0,\; \mathcal L_vz = 0,\; \dots, \;  \mathcal L_v^{(j-1)}z = 0\},$$
where $\mathcal L_v^{(k)}$ stands for the $k$-th iteration of the Lie derivative operator $\mathcal L_v$ in the direction of $v$ (see \cite{K2}). 
The pure stratum $\d_jX^\circ \subset \d_jX$ is defined by the additional constraint  $\mathcal L_v^{(j)}z \neq 0$. The locus $\d_jX$ is the union of two loci: $(1)$ $\d_j^+X$, defined by the constraint  $\mathcal L_v^{(j)}z \geq  0$, and $(2)$ $\d_j^-X$, defined by the constraint  $\mathcal L_v^{(j)}z \leq  0$. The two loci, $\d_j^+X$ and $\d_j^-X$, share a common boundary $\d_{j+1}X$.
\smallskip

Any proto-holographic structure $\mathcal H$ on $Y$ generates some  ``proto-objects", familiar from our study of traversally generic vector fields.   Among them is the stratification $\{\d_jY\}_{1 \leq j \leq n+1}$ of $Y$ by submanifolds of codimension $j$, an analogue of the Morse stratification \hfill\break $\{\d_jX(v)\}_{1 \leq j \leq n+1}$. When $\mathcal H$ is holographic, we can also construct the surrogate ``trajectory space" $\mathcal T(\mathcal H)$ and the partially-defined ``causality map" $C_{\mathcal H}: Y \to Y$. Constructions of such $\mathcal H$-induced objects are explicitly described in the proof of the following lemma. 

\begin{lemma}\label{lem8.10} Every smooth proto-holographic structure $\mathcal H$ on a closed $n$-manifold $Y$ gives rise to a stratification $\{\d_j^\pm Y =_{\mathsf{def}} \d_j^\pm Y(\mathcal H)\}_{1 \leq j \leq n+1}$ of $Y$ by submanifolds  $\d_j^+ Y,\, \d_j^- Y$ of codimension $j$ such that: 
\begin{itemize}
\item $\d_1^+Y \cup \d_1^-Y = Y$,
\item $\d(\d_j^+ Y) = \d(\d_j^- Y) = \d_{j +1}^+Y \cup \d_{j +1}^-Y$ for all $j$. 
\end{itemize}

Every smooth holographic structure $\mathcal H$ on  $Y$ gives rise to:
\begin{itemize}
\item  a causality  map $C_{\mathcal H}: \d_0^+Y \to  \d_0^-Y$ (possibly discontinuous) whose fixed-point set is the locus $\d_2^-Y \cup \dots \cup \d_{n+1}^-Y$,
\item  a  $C_{\mathcal H}$-trajectory space $\mathcal T(\mathcal H)$, the target of a finitely ramified  map $\Gamma_\mathcal H: Y \to \mathcal T(\mathcal H)$. Locally, $\mathcal T(\mathcal H)$ is modeled after the standard cellular complexes $\{\mathsf T_\omega\}_\om$ (as in Theorem 5.3 from \cite{K3}), where $\omega \in  \mathbf\Omega^\bullet_{' \langle n]}$.
\end{itemize}
\end{lemma}

\begin{proof} Consider a proto-holographic atlas $\mathcal H = \{U_\b, h_\b, h_{\b,\b'}\}$ on $Y$. Recall that any point $z_\b^\star = (u_\b^\star, \vec{x}^\b_\star,  \vec{y}^\b_\star) \in \mathsf W_\b \subset  \d\mathsf Z_\b$ has a multiplicity $j(z_\b^\star)$, defined as the multiplicity of the appropriate root $u_\b^\star$ of the $u_\b$-polynomial $P_\b(u_\b, \vec{x}^\b)$. 

In fact, $j(z_\b^\star)$ can be detected just in terms of \emph{the localized maximal cardinality of fibers} of the map $\pi: \mathsf W_\b \to \R^n$, the fibers which are contained in a corresponding connected component of a fiber of $\pi: \mathsf Z_\b \to \R^n$. Here we localize the fibers of $\pi: \mathsf W_\b \to \R^n$ to the vicinity of $z_\b^\star$ in the hypersurface $\mathsf W_\b$.  Indeed, if $j(z_\b^\star)$ is an even number, then any small neighborhood of $z_\b^\star$ contains $\d_{u_\b}$-trajectories in $\mathsf Z_\b$ of the combinatorial type $(122\dots 21)$, where the number of $2$'s is $(j(z_\b^\star) - 2)/2$; if $j(z_\b^\star)$ is an odd number, then any small neighborhood of $z_\b^\star$ contains trajectories in $\mathsf Z_\b$ of the combinatorial type $(122\dots 2)$, where the number of $2$'s is $(j(z_\b^\star) -1)/2$ (see the arguments in the beginning of the proof of Theorem 3.1 from \cite{K4}). Note that the odd and even multiplicities are distinguished by their local topology.

Thus, for any $y \in U_\b \subset Y$, the point $z_\b =_{\mathsf{def}} h_\b(y)$ acquires multiplicity $j_\b(z_\b)$.  By the property of $h_{\b,\b'}$, described in the fourth bullet of Definition \ref{def8.5}, the points in the fibers of $\pi_\b:  h_\b(U_\b \cap U_{\b'}) \to \R^n$, which belong to the same connected component of the corresponding fiber $\pi_\b: \mathsf Z_\b \to \R^n$, are mapped to the points in the fibers of $\pi_{\b'}: h_{\b'}(U_\b \cap U_{\b'}) \to \R^n$, which also belong  to a connected component of the corresponding fiber $\pi_{\b'}: \mathsf Z_{\b'} \to \R^n$. Therefore, employing this property of diffeomorphism $h_{\b,\b'}$ and interpreting the multiplicities in terms of  $\pi$-fibers' cardinalities, we get $j_{\b'}(h_{\b'}(y)) = j_\b(h_\b(y))$ when $y \in U_\b \cap U_{\b'}$.
\smallskip

Consider the stratification  $\{\d_j^\pm  \mathsf W_\b =_{\mathsf{def}} \d_{j}^\pm \mathsf Z_\b(\d_\b)\}_j$ of $\mathsf W_\b$, generated by the constant vector field $\d_\b =_{\mathsf{def}} \frac{\d}{\d u_\b}$. 

We have seen that a  proto-holographic atlas $\mathcal H$ helps to define a stratification $\{\d_jY =_{\mathsf{def}} \d_jY(\mathcal H)\}_{1 \leq j \leq n+1}$ of $Y$ by closed submanifolds: just put $$\d_jY \cap U_\b =_{\mathsf{def}} h_\b^{-1}\big(\d_{j}\mathsf Z_\b(\d_\b) \cap h_\b(U_\b)\big).$$ 

The stratification $\{\d_jY(\mathcal H)\}_{1 \leq j \leq n+1}$ can be refined to a stratification $$\{\d_j^\pm Y=_{\mathsf{def}} \d_jY^\pm(\mathcal H)\}_{1 \leq j \leq n}$$ by compact submanifolds such that $\d_jY = \d_j^+Y \cup \d_j^-Y$ and $\d_{j +1}Y = \d( \d_j^+Y) = \d(\d_j^-Y)$ for all $j$. 

Let us describe its construction. By Theorem \ref{th8.8}, the stratification  $\{\d_j^\pm  \mathsf W_\b =_{\mathsf{def}} \d_{j}^\pm \mathsf Z_\b\}_j$ of $\mathsf W_\b$, generated by the field $\d_\b$, is characterized in terms of the polynomial $P_\b$  as follows: $(u_\b,\vec{x}^\b, \vec{y}^\b) \in \d_{j}^+ \mathsf Z_\b$ if  $\frac{\d^k}{\d u^k}P_\b(u_\b, \vec{x}^\b) = 0$ for all $k < j$, and $\pm\frac{\d^{j}}{\d u^{j}}P_\b(u_\b, \vec{x}^\b) < 0$.  

In fact, the sign of  the derivative $\frac{\d^{j}}{\d u^{j}}P_\b(u_\b, \vec{x}^\b)$ at a root $u_\b$ of multiplicity $j$ can be determined by the simple combinatorial rule: consider the real divisor $D_{P_\b}(\vec{x}^\b)$ of the $u_\b$-polynomial $P_\b(\sim, \vec{x}^\b)$ and decompose its support $\sup(D_{P_\b}(\vec{x}^\b))$ into ``atoms" and ``strings" as in Figure 3. Then any root of even multiplicity, which is an atom, is assigned the polarity ``$-$", any root of even multiplicity which is inside of a string is assigned the polarity ``$+$", any root of odd multiplicity, forming the lower end of a string, is is assigned the polarity ``$+$", while any root of odd multiplicity, forming the upper end of a string, is is assigned the polarity ``$-$"\footnote{these signs are opposite to the signs of the lowest order non-trivial derivatives  $\frac{\d^{j+1}}{\d u^{j+1}}P_\b(u_\b, \vec{x}^\b)$.}. 

\begin{figure}[ht]\label{fig3}
\centerline{\includegraphics[height=1in,width=3.5in]{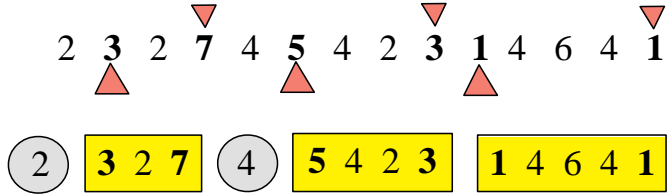}}
\bigskip
\caption{\small{Decomposing the pattern $\om =(23274542314641)$ into atoms (circled) and strings (boxed). The atoms are the isolated points of the locus $\{P_\om \leq 0\}$, and the strings represent the closed intervals of the same locus.}} 
\end{figure}

Therefore, the polarity $\pm$ of a point  $z \in \d_j\mathsf W_\b$ is determined by the ordered sequence of intersection points of the $\d_\b$-trajectory $\g_z \subset \mathsf Z_\b$ trough $z = (u_\b,\vec{x}^\b, \vec{y}^\b)$\footnote{These trajectories $\{\g_z\}$ produce the strings and atoms in $\sup(D_{P_\b}( \vec{x}^\b))$.} with the hypersurface $\mathsf W_\b$, each intersection point $z$ being considered with its multiplicity $j_\b(z)$. As the combinatorial rule above implies, the polarities attached to each intersection $z \in \g_z \cap \mathsf W_\b$ depend only on the multiplicities and the order of points from $\g_z \cap \mathsf W_\b$ along $\g_z$. Again, by the fourth bullet of Definition \ref{def8.5}, this kind of data are preserved under the smooth change of coordinates $h_{\b,\b'}: h_\b(U_\b \cap U_{\b'}) \to h_\b(U_\b \cap U_{\b'})$.  Therefore the local formulas $\{\d_j^\pm Y \cap U_\b =_{\mathsf{def}}  h_\b^{-1}(\d_{j}^\pm\mathsf Z_\b \cap h_\b(U_\b))\}_\b$ produce the desired globally-defined stratifications.
\smallskip

Given holographic atlas $\mathcal H$, let us turn to the construction of the causality map $C_{\mathcal H}: \d_1^+Y \to \d_1^-Y$. Take a point $y \in U_\b \cap \d_1^+Y$  to the point $h_\b(y) \in h_\b(U_\b) = \mathsf W_\b$. Then, in the fiber $\pi_\b^{-1}(h_\b(y))$ of  the map $\pi_\b: h_\b(U_\b) \subset \R\times\R^n \to \R^n$, which contains $h_\b(y)$, take $h_\b(y)$ to the point $z$ with the minimal first coordinate $u_\b(z) > u_b(h_\b(y))$ and such that $P_\b <  0$ in the open interval $(h_\b(y), z) \subset \R\times\R^n$; if no such $z$ exists, put $z = h_\b(y)$. Let $C_{\mathcal H}(y) =_{\mathsf{def}}  h_\b^{-1}(z)$. Since $\mathcal H$ is a holographic structure, the map $C_{\mathcal H}$ is well-defined for all points in $\d_1^+Y \cap U_\b$. 

Note that the model causality map $C_{\d_\b}: \mathsf W_\b^+ \to \mathsf W_\b^-$ has the locus $$\d_2^-\mathsf W_\b \cup \dots \cup \d_{n+1}^-\mathsf W_\b =_{\mathsf{def}}  \d_2^-\mathsf Z_\b \cup \dots \cup \d_{n+1}^-\mathsf Z_\b$$ for its fixed point set.

This local construction of $C_{\mathcal H}(y)$ does not depend on the choice of a chart $U_\b$ which contains $y$. Indeed, let $y \in U_\b \cap U_{\b'}$, $z_\b = h_\b(y)$, and $F_\b(y) = \pi_\b^{-1}\big(\pi_\b(z_\b)\big) \subset \mathsf Z_\b$. Similar notations are introduced for the map $h_{\b'}$. Then, according to the fourth bullet in Definition  \ref{def8.5}, the connecting diffeomorphism $h_{\b,\b'}$ maps each portion $F^\bullet_\b(y)$ of the fiber $F_\b(y)\cap \mathsf W_\b$ that belongs to the same connected component of the fiber $F_\b(y) \subset \mathsf Z_\b$ onto a corresponding portion $F^{'\bullet}_{\b'}(y)$ of the fiber $F'_{\b'}(y) \cap \mathsf W_{\b'}$ that belongs to the same connected component of the fiber $F'_{\b'}(y) \subset \mathsf Z_{\b'}$. Moreover, $h_{\b,\b'}$ respects the $u_\b$-induced order of points in $F^\bullet_\b(y)$ and the $u_{\b'}$-induced order of points in $F^{'\bullet}_{\b'}(y)$. The two locally-defined causality  maps, $$C_{\mathcal H, \b}: Y^+ \cap U_\b \to Y^- \cap U_\b \;\; \text{and} \;\; C_{\mathcal H, \b'}: Y^+ \cap U_{\b'} \to Y^- \cap U_{\b'},$$ operate within the sets $h_\b^{-1}(F^\bullet_\b(y))$ and $h_{\b'}^{-1}(F^{'\bullet}_{\b'}(y))$, respectively. Thus, by employing same combinatorial rules (which depend only on the intrinsically-defined polarities of the corresponding points), by the third bullet in Definition  \ref{def8.5}, we get that $C_{\mathcal H, \b}(y) = C_{\mathcal H, \b'}(y)$.
\smallskip
 
With the given  holographic atlas $\mathcal H$ on $Y$ in place, we declare two points $y_1, y_2 \in Y$ to be equivalent if there exists $y_0 \in Y$ such that both $y_1$ and $y_2$ can be obtained as images of $y_0$ under iterations of the causality  map $C_{\mathcal H}$. We denote by $\mathcal T(\mathcal H)$ the quotient space of $Y$ by this equivalence relation and call it \emph{the trajectory space} of  $\mathcal H$. The obvious finitely-ramified map  $\Gamma_{\mathcal H}: Y \to \mathcal T(\mathcal H)$ helps to define the quotient topology in the trajectory space $\mathcal T(\mathcal H)$. 

Note that, for each $y \in U_\b$, the vicinity of the point $\Gamma_{\mathcal H}(y)$ in $\mathcal T(\mathcal H)$,  is modeled after the standard cell complex $\mathsf T_\omega$ as in Theorem 5.3 from \cite{K3}, where $\omega \in (\omega_\b)_{\preceq} \subset \mathbf\Omega^\bullet_{ ' \langle n]}$. Indeed, the obvious projections $\mathsf W_\b \subset \mathsf Z_\b \to \mathsf T_{\omega_\b}$, with the help of $h_\b^{-1}$, provide for the local models of $\Gamma_{\mathcal H}$.
\hfill 
\end{proof}

\begin{corollary}\label{cor8.14} For any traversally  generic vector field $v$ on $X$, the $\mathbf\Omega^\bullet_ {' \langle n]}$-stratified topological type of the trajectory space $\mathcal T(v)$ depends only on the fillable holographic structure $\mathcal H(v)$ that $v$ generates on the boundary $\d_1X$.  
\end{corollary}

\begin{proof} By Theorem \ref{th8.7}, the fields $v_i$ ($i = 1, 2$), give rise to fillable holographic structures $\mathcal H(v_i)$. By Lemma \ref{lem8.10}, $\mathcal H(v_i)$ canonically produces the causality map $C_{\mathcal H(v_i)}$ (which happen to coincide with the causality  map $C_{v_i}$). By the hypotheses, $\mathcal H(v_1) = \mathcal H(v_2)$, which implies the equality $C_{\mathcal H(v_1)}  = C_{\mathcal H(v_2)}$. Since the spaces $\mathcal T(v_i)$ depend only on $C_{\mathcal H(v_i)} = C_{v_i}$, we conclude that  $\mathcal T(v_1)$ and $\mathcal T(v_2)$ are homeomorphic as $\mathbf\Omega^\bullet_{' \langle n]}$-stratified spaces.
\hfill \end{proof}
\smallskip

\noindent{\bf Remark 2.5.} 
In principle, in accordance with Theorem \ref{th8.7}, a given fillable holographic  structure $\mathcal H$ on $Y$, may have many realizations $\{\mathcal H(v)\}$ by traversally  generic fields $v$ on a variety of $X$'s whose boundary is $Y$. However, Holography Theorem \ref{th1.1} claims that the topological (often smooth) type of $X$, together with the $v$-induced oriented foliation $\mathcal F(v)$, is determined by $\mathcal H$ via the causality map $C_\mathcal H = C_v$. \hfill $\diamondsuit$

\begin{definition}\label{def8.8} Let $\mathcal H$ be a smooth holographic structure on $Y$ and $\mathcal T(\mathcal H)$ its trajectory space. 

A continuous function $h: \mathcal T(\mathcal H) \to \R$ is called \emph{holographically smooth}, if the composite function $Y \stackrel{\Gamma_{\mathcal H}}{\rightarrow} \mathcal T(\mathcal H) \stackrel{h}{\rightarrow} \R$ is smooth on $Y$. \smallskip

We denote by $C^\infty(\mathcal T(\mathcal H))$ the algebra of holographically smooth functions on $\mathcal T(\mathcal H)$. Via the induced monomorphism\footnote{Indeed, the map $Y \to \mathcal T(\mathcal H)$ is onto.} $\Gamma_{\mathcal H}^\ast$, the algebra $C^\infty(\mathcal T(\mathcal H))$ maps onto a subalgebra $C^\infty_\mathcal H(Y)$ of $C^\infty(Y)$. \hfill $\diamondsuit$
\end{definition}


\begin{corollary}\label{cor8.15} Let  $\mathcal H$ be a holographic structure on a closed smooth manifold $Y$. 

Then $\mathcal H$ determines the algebra $C^\infty(\mathcal T(\mathcal H))$ of holographically smooth functions on the trajectory space $\mathcal T(\mathcal H)$ as a subalgebra $C^\infty_\mathcal H(Y)$ of $C^\infty(Y)$. \smallskip

Moreover, any diffeomorphism $\Phi: Y \to Y$ induces an algebra isomorphism $$\Phi^\ast: C^\infty_\mathcal H(Y) \to C^\infty_{\Phi^\ast\mathcal H}(Y),$$ where $\Phi^\ast\mathcal H$ denotes the $\Phi$-induced holographic structure.
\end{corollary}

\begin{proof} By Lemma \ref{lem8.10}, the obvious maps $\Gamma_{\mathcal H}: Y \to \mathcal T(\mathcal H)$ and $\Gamma_{\Phi^\ast\mathcal H}: Y \to \mathcal T(\Phi^\ast\mathcal H)$ are defined, via the proto-causality  maps $C_{\mathcal H}$ and $C_{\Phi^\ast\mathcal H}$, in terms of the structures $\mathcal H$ and $\Phi^\ast\mathcal H$, respectively. 

It is on the  level  of definitions to check that $\Phi \circ \Gamma_{\Phi^\ast\mathcal H} =  \Gamma_{\mathcal H}$. Therefore, using that $\Phi$ is a diffeomorphism, we get that $\Phi^\ast: C^\infty_\mathcal H(Y) \to C^\infty_{\Phi^\ast\mathcal H}(Y)$ is an isomorphism of algebras. 
\hfill \end{proof}
\smallskip

\begin{question}\label{q8.4} Does the pair of algebras $C^\infty_\mathcal H(Y) \subset C^\infty(Y)$ determine the holographic structure $\mathcal H$ on $Y$? \hfill $\diamondsuit$
\end{question}

\smallskip


Now let us combine Holography Theorem \ref{th1.1} with Theorem \ref{th8.7} into a single proposition. 

\begin{theorem}\label{th8.8}{\bf (The Holographic Principle)} \hfill\break
\begin{itemize}
\item For a traversally generic vector field $v$ on a (oriented) compact connected smooth manifold $X$, the topological type of $(X, \mathcal F(v))$ is  determined by the $v$-induced (orientably) fillable holographic structure $\mathcal H(v)$ on $\d X$. \smallskip

\item Conversely, any (orientably) fillable holographic structure $\mathcal H$ on a closed manifold $Y$ is of the form $\mathcal H(v)$ for a traversally generic field $v$ on some (orientable) manifold $X$ such that $\d X = Y$. Moreover, the topological type of the pair $(X, \mathcal F(v))$ is determined by $\mathcal H$. \smallskip

\item When, in addition, $v$ is as in the second bullet of Theorem \ref{th1.1}, then $\mathcal H(v)$ determines the smooth topological type of $(X, \mathcal F(v))$.\smallskip

\item Conversely, if a fillable holographic structure $\mathcal H$ on $Y$ is based on the combinatorial types $\omega \notin (33)_\succeq \cup (4)_\succeq$, then $\mathcal H$ determines the smooth topological type of the pair $(X, \mathcal F(v))$, where $\d X = Y$ and $\mathcal H(v) = \mathcal H$.
\end{itemize}
\end{theorem}

\begin{proof} By Theorem \ref{th8.7}, any traversally  generic field $v$ on $X$ gives rise to a fillable holographic structure $\mathcal H(v)$ on $\d X$.

On the other hand, by the same Theorem \ref{th8.7}, any given fillable holographic structure $\mathcal H$  on a closed manifold $Y$ produces a traversally  generic vector field $v$ on some $X$ such that $\d X = Y$. 

By Lemma \ref{lem8.10}, the causality map $C_v$ can be described intrinsically in terms of $\mathcal H(v)$ as the map $C_{\mathcal H(v)}: \d_1^+X(\mathcal H(v)) \to \d_1^-X(\mathcal H(v))$.  By Holography Theorem \ref{th1.1},  $C_v$ allows for a reconstruction of the topological type of the pair $(X, \mathcal F(v))$. Therefore, $C_{\mathcal H(v)}$ does the same job.

Finally, $v$ is as in the second bullet of Theorem \ref{th1.1}, if and only if its trajectories have the combinatorial types $\omega \notin (33)_\succeq \cup (4)_\succeq$. Thus  if a fillable $\mathcal H$ is based on the combinatorial types $\omega \notin (33)_\succeq \cup (4)_\succeq$, then any $\mathcal H$-generated $v$ is as in the second bullet of Theorem \ref{th1.1} and hence its causality map $C_v$ allows for a reconstruction of the smooth type of $(X, \mathcal F(v))$. 
\hfill
\end{proof}

\section{Holographic Structures on Homotopy Spheres}

The smooth structures on homotopy spheres have been a favorite subject in topology for a number of years. To support the tradition, we will make few superficial observations about the holographic structures on \emph{homotopy spheres} and their relations to the underlying smooth structures. Stated loosely, our main remark here is that the universe of fillable holographic structures on smooth homotopy spheres maps surjectively on the universe the exotic spheres, so that the map is a homomorphism of the semigroups. \smallskip

In the end of Section 1, we have noticed that the universe of fillable holographic structures even on the standard sphere $S^n$ is at least as rich as the universe of smooth types of closed $(n+1)$-manifolds! \smallskip

Let us recall one old appealing result of Wall \cite{Wa}, the classification of smooth highly connected even-dimensional manifolds:

\begin{theorem}\label{th8.9}{\bf (Wall)} For $n \geq 3$ the diffeomorphism class of $(n - 1)$-connected $2n$-manifold $X$ with boundary a homotopy sphere are in natural bijections with the isomorphism classes of $\Z$-valued non-degenerate $(-1)^n$-symmetric bilinear forms with a quadratic refinement in $\pi_n(BSO(n))$\footnote{$\pi_n(BSO(n)) \approx \Z\; \;\text{if}\; n \equiv 0 (2)$, and $\pi_n(BSO(n)) \approx \Z_2\;  \;\text{if}\; n \equiv 1 (2)$.}. 
\hfill $\diamondsuit$
\end{theorem}

\noindent{\bf Remark 3.1.} 
Recall that, if a homotopy sphere $\Sigma^{2n-1}$, $n \geq 3$, bounds a \emph{framed}  smooth $(2n)$-manifold $Y$, then by surgery, $\Sigma^{2n-1}$ bounds a framed $(n-1)$-connected manifold $X$ (cf. \cite{Kos}, Theorem 2.2 on page 201). Therefore, all such spheres satisfy the hypotheses of Theorem \ref{th8.9}.  \hfill $\diamondsuit$
\smallskip

Theorem \ref{th8.8} implies the following result.

\begin{corollary}\label{cor8.16} Let $\mathcal H^{\{\mathsf{fill}\}}\mathsf \Sigma_{2n-1}$ denote the set of orientably fillable holographic structures on smooth homotopy $(2n-1)$-spheres.
\smallskip

For $n \geq 3$, the set $\mathcal H^{\{\mathsf{fill}\}}\mathsf \Sigma_{2n-1}$ maps to the set of isomorphism classes of non-degenerate $\Z$-valued $(-1)^n$-symmetric bilinear forms on free finite-dimensional $\Z$-modules.

Any non-degenerate bilinear form $\Phi: M \otimes M \to \Z$ on a free $\Z$-module $M$, that admits a quadratic refinement $\mu$ as in $(\ref{eq8.13})$, is in the image of that map.
\end{corollary}

\begin{proof} Let $\mathcal H$ be an orientably fillable holographic structure on a homotopy sphere $\Sigma^{2n-1}$. By Theorem \ref{th8.8}, such $\mathcal H$ is of the form $\mathcal H(v)$ for some traversally  generic $v$ on an oriented smooth manifold $X$ whose boundary is $\Sigma^{2n-1}$.  Moreover, the topological type of $X$ is determined by $\mathcal H = \mathcal H(v)$. 

Let $$\Phi_\mathcal H: \big(H_n(X; \Z)/\textup{Tor}) \otimes \big(H_n(X; \Z)/\textup{Tor}\big) \to \Z$$ be the bilinear $(-1)^n$-symmetric intersection form on $X$. Since $\d X$ is a homotopy sphere, the form $\Phi_\mathcal H$ is non-degenerate over $\Z$.

Thus we have mapped any orientably fillable holographic structure $\mathcal H$ on a homotopy $(2n-1)$-sphere to the bilinear form $\Phi_\mathcal H$. \smallskip

Let us show that any bilinear $(-1)^n$-symmetric \emph{non-degenerate over} $\Z$ form $\Phi: M \otimes M \to \Z$ on a finitely-dimensional free $\Z$-module $M$ can be produced this way, provided that $\Phi$ is enhanced by a quadratic form $\mu: M \to \Z/\langle1 - (-1)^n\rangle$. Recall that the relations between $\Phi$ and $\mu$ are described by the formulas: 
\begin{eqnarray}\label{eq8.13}  
\Phi(a, a) & = & \mu(a) + (-1)^n \mu(a),  \nonumber \\
\Phi(a, b) & = & \mu(a+b) -\mu(a) -\mu(b)\;\;\; \textup{mod}\; \langle1 - (-1)^n\rangle 
\end{eqnarray}
for any $a, b \in M$. Note that the RHS of the formula for $\Phi(a, a)$ is well-defined as an element of $\Z$.

By Theorem \ref{th8.9}, any such pair $(\Phi, \mu)$ is realized  on some smooth orientable $(n - 1)$-connected $2n$-manifold $X$ which bounds  a homotopy sphere $\Sigma^{2n-1}$. For such $X$, $H_n(X; \Z) \approx M$ and $\Phi$ becomes the intersection form. The quadratic form $\mu: H_n(X; \Z) \to \Z/\langle1 - (-1)^n\rangle$ is defined by the self-intersections of immersed spheres $S^n \propto X$. 

By Corollary 4.1 from \cite{K1} and Theorem 3.5 from \cite{K2}, $X$ admits a traversally  generic vector field $v$. By Theorem \ref{th8.8}, $v$ induces an orientably fillable holographic  structure $\mathcal H(v)$ on $\Sigma^{2n-1}$. Therefore $\Phi = \Phi_{\mathcal H(v)}$.
\hfill \end{proof} 

Let $X_1$ bounds a homotopy $k$-sphere $\Sigma_1$, and $X_2$ a homotopy $k$-sphere $\Sigma_2$. Consider two traversally  generic fields: $v_1$ on $X_1$ and $v_2$ on $X_2$.  By attaching a $1$-handle $\mathbf H =_{\mathsf{def}} D^1 \times D^k$ to $\Sigma_1 \coprod \Sigma_2$, we form a new manifold $X =_{\mathsf{def}} X_1 \#_\d X_2$ that bounds the homotopy sphere $\Sigma =_{\mathsf{def}} \Sigma_1 \# \Sigma_2$. The $1$-surgery can be performed so that the fields $v_1$ and $v_2$ extend over the $1$-handle $\mathbf H$ to form a traversally  generic field $v$ on $X$ such that the following diffeomorphisms are valid: 
\begin{itemize}
\item $\d_1^+ X(v) \approx  \d_1^+ X_1(v_1) \coprod \big(\d_1^+ X_2(v_2) \setminus D^k\big)$, \smallskip

\item  $\d_1^- X(v) \approx  \big(\d_1^- X_1(v_1) \setminus D^k\big) \coprod \d_1^-X_2(v_2),$  \smallskip

\item $\d_2^- X(v) \approx  \d_j^- X_1(v_1) \coprod \d_j^- X_2(v_2)$,  \smallskip

\item $\d_2^+ X(v) \approx  \d_j^+X_1(v_1) \coprod \d_j^+ X_2(v_2) \coprod S^{k-1}$,  \smallskip

\item $\d_j^\pm X(v) \approx  \d_j^\pm X_1(v_1) \coprod \d_j^\pm X_2(v_2)$ for all $j > 2$. 
\end{itemize}

These formulas describe that the handle $\mathbf H$ has a bottleneck with respect to the $v$-flow. The construction of the field $v$ uses a $v_1$-adjusted regular neighborhood $U_1$  of a $v_1$-trajectory of the combinatorial type $(11)$ and a $v_2$-adjusted regular neighborhood $U_2$  of a $v_2$-trajectory of the combinatorial type $(11)$. Outside of $U_1\coprod U_2$ the field $v$ coincides with $v_1 \coprod v_2$.

This boundary connected sum construction produces a holographic cobordant to $\emptyset$ structure $\mathcal H(v) =_{\mathsf{def}} \mathcal H(v_1) \#_\d \mathcal H(v_2)$ on $\Sigma$. 

\begin{lemma}\label{lem8.11} Let $\mathcal H(v_i)$ be an orientably fillable holographic  structure on a homotopy \hfill\break $(2n-1)$-sphere $\Sigma_i$, $i = 1,2$. 

Then the bilinear $\Z$-form $\Phi_{\mathcal H(v)}$ that, according to Corollary \ref{cor8.16}, corresponds to the orientably fillable holographic structure $\mathcal H(v) =_{\mathsf{def}} \mathcal H(v_1) \#_\d \mathcal H(v_2)$ on the homotopy sphere $\Sigma =_{\mathsf{def}} \Sigma_1 \# \Sigma_2$ is the direct sum $\Phi_{\mathcal H(v_1)} \oplus \Phi_{\mathcal H(v_2)}$of the forms $\Phi_{\mathcal H(v_1)}$ and $\Phi_{\mathcal H(v_2)}$. \hfill\qed
\end{lemma}

\begin{proof} Employing the construction from the proof of Corollary \ref{cor8.16}, the validation of the lemma is at the level of definitions.
\hfill \end{proof}
\smallskip

We leave to the reader to discover a generalization of the boundary connected sum construction $\mathcal H(v) =_{\mathsf{def}} \mathcal H(v_1) \#_\d \mathcal H(v_2)$ for any pair of fillable holographic structures. This generalization will produce a holographic structure $\mathcal H =_{\mathsf{def}} \mathcal H_1 \#_\d \mathcal H_2$ on $\Sigma=_{\mathsf{def}} \Sigma_1 \# \Sigma_2$, utilizing two given holographic structures: $\mathcal H_1$ on $\Sigma_1$ and $\mathcal H_2$ on $\Sigma_2$. Each of these two structures $\mathcal H_i$ ($i = 1, 2$) is enhanced by a choice of a pair of ``base" points $s^+_i, s^-_i \in \Sigma_i$. Both points are contained in some element $U_i^{\textup{base}}$ of the combinatorial type $(11)$ from the atlas $\mathcal H_i$;  their images, $h_i^{\textup{base}}(s^+_i)$ and $h_i^{\textup{base}}(s^-_i)$, under the atlas homeomorphism $h_i^{\textup{base}}: U_i^{\textup{base}} \to \R^n \times \R$, project to the same point in $\R^n$.   \smallskip

Note that the trajectory space $\mathcal T(\mathcal H_1 \#_\d \mathcal H_2)$ is obtained from   the spaces $\mathcal T(\mathcal H_1)$  and $\mathcal T(\mathcal H_2)$ by gluing them along some $n$-disks $D^n_1$ and $D^n_2$, the disk $D^n_i$  resides in the maximal (i.e., indexed by $\omega = (11)$) stratum of $\mathcal T(\mathcal H_i)$.
\smallskip

The operation
\begin{eqnarray}\label{eq8.14} 
(\mathcal H_1; s^+_1, s^-_1),\; (\mathcal H_2; s^+_2, s^-_2)\; \Rightarrow \;(\mathcal H_1 \#_\d \mathcal H_2; s^+_1, s^-_2)
\end{eqnarray}
converts the set of all \emph{based} holographic structures on smooth homotopy $n$-spheres 
into a semigroup (monoid)  $\mathcal H\mathsf \Sigma_n$. The associativity of the operation in (\ref{eq8.14}) is evident.\smallskip

The order of two holographic structures $\mathcal H_1$ and $\mathcal H_2$ in the operation ``$\#_\d$" does matter: we think of $\Sigma_2$ as residing ``above" $\Sigma_1$, at least in the vicinity of points $s_1^- \in \Sigma_1$ and $s_2^+ \in \Sigma_2$ where the ``virtual" $1$-handle $\mathbf H$, responsible for the operation, is attached.

It seems that the holographic structure $\mathcal H_1 \#_\d \mathcal H_2$ indeed may depend on the choice of base pairs $(s^+_1, s^-_1)$ and $(s^+_2, s^-_2)$ and may change at least when they have been chosen in different connected components of the pure strata of the combinatorial type $(11)$ in $\Sigma_1$ or in $\Sigma_2$. 
\smallskip

The reader should be aware: now we are entering The Speculation Zone... \smallskip

Employing the connected sums of based holographic structures on homotopy $n$-spheres, we can form words, representing elements of $\mathcal H\mathsf \Sigma_n$ or of $\mathcal H^{\{\mathsf{fill}\}}\mathsf \Sigma_n$  in alphabets to be discovered and understood...  
\smallskip

We denote by $\mathcal HS_n$ the sub-semigroup of based holographic structures on the standard $n$-sphere $S^n$. 
We also will employ a smaller semigroup $\mathcal H^{\{\mathsf{fill}\}}S^n$ (still gigantic), based on fillable holographic structures $\mathcal H(v)$ on $S^n$. We denote by $\mathcal H^{\{\mathsf{ball-fill}\}}S^n$ sub-semigroup of holographic structures  that are induced by traversally  generic fields $v$ on the standard ball $B^{n+1}$. 
Unfortunately, we do not know much about ``size" of the quotients $\mathcal HS^n\big/ \mathcal H^{\{\mathsf{fill}\}}S^n$ or $\mathcal H^{\{\mathsf{fill}\}}S^n\big/\mathcal H^{\{\mathsf{ball-fill}\}}S^n$. We speculate that $\mathcal H^{\{\mathsf{fill}\}}S^n\big/\mathcal H^{\{\mathsf{ball-fill}\}}S^n$ is as rich as the set of smooth types of $(n+1)$-dimensional closed manifolds. 
\smallskip

Note that  $\mathcal HS_n$ acts, via the connected sum operation $\#_\d$, on $\mathcal H\mathsf \Sigma_n$ from the left and from the right, or rather from ``above"  and from ``below". Similarly, $\mathcal H^{\{\mathsf{fill}\}}S^n$ and $\mathcal H^{\{\mathsf{ball-fill}\}}S^n$ act, via the connected sum operation $\#_\d$, on $\mathcal H^{\{\mathsf{fill}\}}\mathsf \Sigma_n$ from the left and from the right.
\smallskip

Here are two informal questions that motivate us: \smallskip

\emph{What part of a given holographic structure on a closed manifold $Y$ can be localized to the ``spherical nuggets", representing an element of  $\mathcal HS^n$?}  
\smallskip

\emph{What part of a given fillable holographic structure on a closed manifold $Y$ can be localized to the ``spherical nuggets", representing an element of  $\mathcal H^{\{\mathsf{fill}\}}S_n$?} An element of $\mathcal H^{\{\mathsf{ball-fill}\}}S^n$? \smallskip

\begin{definition}\label{def8.9}  We introduce \emph{two} sets of relations among the elements of the semigroups $\mathcal H\mathsf \Sigma_n$, $\mathcal H^{\{\mathsf{fill}\}}\mathsf \Sigma_n$:  

\begin{itemize}
\item If a based holographic structure $\mathcal H$ on some homotopy $n$-sphere $\Sigma$ is a connected sum of a based holographic structure $\mathcal H'$ on $\Sigma$ and a based  holographic structure on $S^n$, then we declare $\mathcal H$ \emph{stably equivalent} to $\mathcal H'$.\smallskip

We denote by $\mathcal H\mathsf{\Sigma}_n^{\textup{st}}$ the set of based stably equivalent holographic structures on the homotopy $n$-spheres. 
\smallskip

\item Similarly, if a based fillable holographic structure $\mathcal H$ on some homotopy $n$-sphere $\Sigma$ is a connected sum of a based fillable holographic structure $\mathcal H'$ on a $\Sigma$ and a based (fillable) holographic structure on $S^n$ 
then we declare $\mathcal H$ \emph{stably equivalent} to $\mathcal H'$.\smallskip

We denote by $\mathcal H^{\{\mathsf{fill}\}}\mathsf{\Sigma}_n^{\textup{st}}$ the set of based strongly stably equivalent fillable holographic structures on the homotopy $n$-spheres. \hfill $\diamondsuit$
\end{itemize}
\end{definition}

The stable equivalence classes of based (fillable) holographic structures on homotopy spheres are amenable to the connected sum operation. Therefore, these equivalences give rise to the obvious epimorphisms of semigroups: 

\begin{eqnarray}\label{eq8.15}
\mathcal H\mathsf \Sigma_n \to \mathcal H\mathsf{\Sigma}_n^{\textup{st}}, \qquad
\mathcal H^{\{\mathsf{fill}\}}\mathsf \Sigma_n \to \mathcal H^{\{\mathsf{fill}\}}\mathsf{\Sigma}_n^{\textup{st}}
\end{eqnarray}

Let $\mathsf{\Theta}_n$ be the commutative semigroup of smooth homotopy $n$-spheres with the addition, induced by the connected sum operation. For $n \neq 4$, thanks to the $h$-cobordism theorem (see \cite{Mi}, \cite{KM}), $\mathsf{\Theta}_n$ is an abelian group. The role of opposite element $-[\Sigma^n]$ is played by the homotopy sphere $\Sigma^n$, taken with the opposite orientation.

\begin{lemma}\label{lem8.12} For $n \neq 4$, under the connected sum operation, the sets $\mathcal H\mathsf{\Sigma}_n^{\textup{st}}$ and $\mathcal H^{\{\mathsf{fill}\}}\mathsf{\Sigma}_n^{\textup{st}}$ are groups. 
\end{lemma}

\begin{proof} Given a based holographic structure $(\mathcal H;\, s^+, s^-)$ on $\Sigma^n$, represented by an atlas $\mathcal A =_{\mathsf{def}} \{h_\b: U_\b \to \R^n \times \R\}$, consider the based holographic structure $(^\neg\mathcal H;\,  s^+, s^-)$ represented by the same atlas $\mathcal A$ on $\Sigma^n$ with the orientation of $\Sigma^n$ being reversed. The holographic structure $\mathcal H \#_\d\, ^\neg\mathcal H$ is defined on the homotopy sphere $\Sigma^n \# - \Sigma^n$. By the $h$-cobordism theorem, $\Sigma^n \# (- \Sigma^n) \approx S^n$ for any $n \neq 4$. Thus the holographic structure $\mathcal H \#_\d\, ^\neg\mathcal H$ is defined on $S^n$, and thus is trivial in $\mathcal H\mathsf{\Sigma}_n^{\textup{st}}$. 

The same argument applies to the based fillable holographic structures on the homotopy spheres.  
\hfill 
\end{proof}

\begin{corollary}\label{cor8.17} For $n \neq 4$, the forgetful map $\mathcal K_n^{\textup{st}}: \mathcal H^{\{\mathsf{fill}\}}\mathsf{\Sigma}_n^{\textup{st}} \to \mathsf{\Theta}_n$ that takes each pair \hfill \break (an exotic sphere, the stable equivalence class of a fillable holographic structure on it) to the exotic sphere is a group epimorphism.
\end{corollary}

\begin{proof} Smooth homotopy spheres are stably parallelizable manifolds (see \cite{Kos}, Corollary 8.6). So, their Pontrjagin  and Stiefel-Whitney characteristic classes vanish. Therefore any homotopy sphere $\Sigma^n$ bounds an oriented manifold $X$ (see \cite{CF}). 

By Theorem 3.5 from \cite{K2}, $X$ admits a traversally  generic vector field $v$ which generates a fillable holographic  structure $\mathcal H(v)$ on $\Sigma^n$. As a result, any homotopy sphere admits a fillable holographic structure. By piking a pair of based points, we convert this structure to a based one. Thus the forgetful map $\mathcal K_n^{\textup{st}}$ is surjective. By Lemma \ref{lem8.12}, it is a homomorphism of groups.
\hfill 
\end{proof}

Thanks to the pioneering works of Kervaire \& Milnor \cite{KM}, Browder \cite{Br}, and recently, of Hill, Ravenel, \& Hopkins \cite{HHR}, the groups $\mathsf{\Theta}_n$, $n \neq 4$, are well-understood\footnote{Presently, the only remaining ``problematic" dimensions are $n = 4$ and $n = 126$.} in terms of the stable homotopy groups of spheres $\pi_n^{\textup{st}}$ and of the homomorphism $$J: \pi_n(SO(N)) \to \pi_{n+N}(S^N),$$ where $N \gg n$. 

So, for $n \neq 4$, the group of stably equivalent holographic structures on homotopy $n$-spheres is at least as rich as its abelian image  $\mathsf{\Theta}_n$. \smallskip

Against the background of our investigations of traversally generic flows, the kernels of the forgetful maps $\mathcal K_n^{\textup{st}}: \mathcal H\mathsf{\Sigma}_n^{\textup{st}} \to \mathsf{\Theta}_n$ and $\mathcal K_n: \mathcal H\mathsf{\Sigma}_n \to \mathsf{\Theta}_n$ are interesting objects to study. \smallskip

Another challenging and interesting problem is to say something about the structure of the sizable quotient $\mathcal H^{\{\mathsf{fill}\}}S^n\big/\mathcal H^{\{\mathsf{ball-fill}\}}S^n$.


\begin{thebibliography}{30}

\bibitem [Bo]{Bo} Boardman, J. M., {\it Singularities of Differentiable Maps}, Publ. Math. I.H.E.S. 33 (1967), 21-57.

\bibitem [Br]{Br} Browder, W.,(1969), {\it The Kervaire invariant of framed manifolds and its generalization}, Annals of Mathematics 90 (1): 157�186.

\bibitem [CF]{CF} Conner, P.E., Floyd, E.E., {\it Differentiable Periodic Maps}, Springer-Verlag, 1964.

\bibitem [HHR]{HHR}  Hill, M., Hopkins, M., Ravenel, D., {\it On non-existence of elements of Kervaire invariant one}, arXiv:0908.3724 [math.AT].

\bibitem [K1]{K1} Katz, G., {\it Stratified  Convexity \& Concavity of Gradient Flows on Manifolds with Boundary}, Applied Mathematics, 2014, vol. 5, 2823-2848. \underline{http://www.scirp.org/journal/am}

\bibitem[K2]{K2} Katz, G., {\it  Traversally Generic \& Versal Flows: Semi-algebraic Models of Tangency to the Boundary}, Asian J. of Math., vol. 21, No. 1 (2017), 127-168 (arXiv:1407.1345v1 [mathGT] 4 July, 2014)).

\bibitem[K3]{K3} Katz, G., {\it The Stratified Spaces of Real Polynomials \& Trajectory Spaces of Traversing Flows}, JP J. of Geometry and Topology, v. 19, No. 2 (2016), 95-160 (arXiv:1407.2984v3 [mathGT] 6 Aug 2014).

\bibitem[K4]{K4} Katz, G., {\it Causal Holography of Traversing Flows}, arXiv:1409.0588v1 [mathGT] (2 Sep 2014).

\bibitem [KM]{KM} Kervaire, M. A., Milnor, J. W., {\it Groups of homotopy spheres: I}. Annals of Mathematics (Princeton University Press) 77 (3): 504�537 (1963).

\bibitem [Kos]{Kos} Kosinski, A., {\it Differential Manifolds}, Academic Press, Boston, San Diego, New York, London, Sydney, Tokyo, Toronto, 1992.

\bibitem [Mi]{Mi} Milnor, J., {\it Lectures on the $h$-cobordism Theorem}, Princeton University Press, Princeton, New Jersey, 1965.

\bibitem [Mo]{Mo} Morse, M., {\it Singular points of vector fields under  general boundary conditions}, Amer. J. Math. 51 (1929), 165-178.

\bibitem [Wa]{Wa} Wall, C.T.C., {\it Classification of $(n -1)$-connected $2n$-manifolds}, Ann. of Math. 75 (1962),163-189.

\end{thebibliography}
\end{document}